%% file: Preprint-28.tex
\documentclass[12pt,oneside,openany,article]{memoir}
\usepackage{mempatch}

\nouppercaseheads
\usepackage[amsthm,thmmarks,hyperref]{ntheorem}
\usepackage[noTeX]{mmap}
\usepackage{etex}
\usepackage{todonotes}
\usepackage[english]{babel}
\usepackage[leqno]{mathtools}

\usepackage{mathrsfs}
\usepackage{upgreek}
\usepackage[compress,square,comma,numbers]{natbib}
\usepackage{hypernat} 

\usepackage{sfmath}

\usepackage{microtype}

\usepackage{subfig}
\usepackage{wrapfig}
\usepackage{array}

\usepackage{graphicx,color}
\definecolor{gray}{gray}{0}
\pagecolor{white}

\usepackage{hyperxmp}
\usepackage{xr-hyper}
\usepackage{nameref}
\usepackage[pdftex,bookmarks,pdfnewwindow,plainpages=false,unicode]{hyperref}

\usepackage{bookmark}

\usepackage{enumitem}

\usepackage{amssymb} 

\hypersetup{
colorlinks=true,
linkcolor=black,
citecolor=black,
urlcolor=blue,
pdfauthor={Victor Ivrii},
pdftitle={Spectral Asymptotics for Fractional Laplacians},
pdfsubject={Sharp Spectral Asymptotics},
pdfkeywords={Microlocal Analysis,  Sharp Spectral Asymptotics, Periodic flows},
bookmarksdepth={4}
}

\numberwithin{equation}{chapter}

\theoremstyle{plain}
\newtheorem{theorem}{Theorem}[chapter]
\newtheorem{lemma}[theorem]{Lemma}
\newtheorem{proposition}[theorem]{Proposition}
\newtheorem{corollary}[theorem]{Corollary}

\theoremstyle{definition}

\newtheorem{Problem}[theorem]{Problem}
\theoremstyle{remark}
\newtheorem{remark}[theorem]{Remark}

\DeclareMathAlphabet{\mathpzc}{OT1}{pzc}{m}{it}

 \newcommand{\cX}{\mathcal{X}}

 \newcommand{\sC}{\mathscr{C}}
 
 \newcommand{\sE}{\mathscr{E}}

 \newcommand{\sH}{\mathscr{H}}

 \newcommand{\sL}{\mathscr{L}}

\newcommand{\Weyl}{{\mathsf{Weyl}}}

\newcommand{\D}{{\mathsf{D}}}

\newcommand{\N}{{\mathsf{N}}}

\newcommand{\W}{{{\mathsf{W}}}}

\newcommand{\dist}{{{\mathsf{dist}}}}

\newcommand{\bC}{{\mathbb{C}}}

\newcommand{\bR}{{\mathbb{R}}}

\newcommand{\bZ}{{\mathbb{Z}}}

\newcommand{\fD}{{\mathfrak{D}}}

\def\1{\boldsymbol {|}}
\newcommand{\3}{{|\!|\!|}}
\newcommand{\Def}{\mathrel{\mathop:}=}

\newcommand{\Ker}{\operatorname{Ker}}

\newcommand{\Ran}{\operatorname{Ran}}

\renewcommand{\Re}{\operatorname{Re}}       

\newcommand{\supp}{\operatorname{supp}}

\newcommand{\Tr}{\operatorname{Tr}}

\newcommand{\Vol}{\operatorname{vol}}

\newenvironment{claim*}[1][{}]{\vglue10pt
\begin{trivlist}
\item[{\hskip\labelsep#1}]}{\vglue10pt\end{trivlist}}

\newcounter{note}

\DeclareTextCommand{\textbeta}{PU}{\83\262}
\DeclareTextCommand{\textmu}{PU}{\80\265}
\DeclareTextCommand{\texttau}{PU}{\83\304}

\DeclareTextCommand{\textlesssim}{PU}{\9042\162}
\DeclareTextCommand{\textgtrsim}{PU}{\9042\163}

\DeclareTextCommand{\textpartial}{PU}{\9042\002}
\DeclareTextCommand{\texttwosuperior}{PU}{\80\262}
\DeclareTextCommand{\textGamma}{PU}{\83\223}
\DeclareTextCommand{\textxinferior}{PU}{\9040\223}
\DeclareTextCommand{\textiinferior}{PU}{\9035\142}
\DeclareTextCommand{\textjinferior}{PU}{\9054\174}

\DeclareTextCommand{\textge}{PU}{\9042\145}
\DeclareTextCommand{\textle}{PU}{\9042\144}
\DeclareTextCommand{\texthat}{PD1}{\136}

\begin{document}
\title{Spectral Asymptotics for Fractional Laplacians}
\author{Victor Ivrii}

\maketitle
{\abstract%
In this article we consider fractional Laplacians which seem to be of interest to probability theory. This is a rather new class of operators for us but our methods works (with a twist, as usual). Our main goal is to derive a two-term asymptotics as one-term asymptotics is easily obtained by R.~Seeley's method.
\endabstract}


In this section we consider fractional Laplacians.

This is a rather new class of operators for us but our methods works (with a twist, as usual). Our main goal is to derive a two-term asymptotics as one-term asymptotics is rather easily obtained by R.~Seeley's method.

\chapter{Problem set-up}
\label{sect-8-5-1}

Let us consider a bounded domain $X\subset \bR^d$ with the smooth boundary $\partial X\in \sC^\infty$\,\footnote{\label{foot-8-22} Alternatively consider a bounded domain $X$ on the Riemannian manifold $\cX$.}. In this domain we consider a fractional Laplacian
$\Lambda_m=(\Delta^{m/2})_\D$ with $m>0$ originally defined on functions
$u\in \sC^\infty (\bR^d): u|_{\partial X }=0$ by
\begin{equation}
\Lambda_{m,X}\Def  (\Delta^{m/2})_\D u =R _X \Delta^{m/2} (\uptheta_X u)
\label{8-5-1}
\end{equation}
where $\uptheta_X$ is a characteristic function of $X$, $R_X$ is an operator of restriction to $X$ and $\Delta^{m/2}$ is a standard pseudodifferential operator in $\bR^d$ with the Weyl symbol $g(x,\xi)^{m/2}$ where as usual $g$ is non-degenerate Riemannian metrics.

\begin{remark}\label{rem-8-5-1}
\begin{enumerate}[label=(\roman*), wide, labelindent=0pt]
\item\label{rem-8-5-1-i}
We consider $\Lambda_{m,X}$ as an unbounded operator in $\sL^2(X)$ with domain
$\fD(\Lambda_{m,X})=\{u\in \sL^2(\bR^d):\,\supp (u) \subset \bar{X}, R_X\Lambda_m \in \sL^2(X)\}\subset \sH_0^{m/2} (X)$.

\item\label{rem-8-5-1-ii}
This operator can also be introduced through positive quadratic form with domain $\{ u\in \sH^m(\bR^d), \, \supp(u)\subset \bar{X}\}$ and is a positive self-adjoint operator  which is Friedrichs extension of operator originally defined on $\sH_0^{m}(X)$.

\item\label{rem-8-5-1-iii}
We can consider this operator as a bounded operator from $\sH_0^{m/2}(X)$ to
$\sH^{-m/2}(X)\Def \sH_0^{-m/2\,*}(X)$.

\item\label{rem-8-5-1-iv}
Let $0< m\notin 2\bZ$. Then $\fD(\Lambda_{m,X})\subset \sH^m (X)$ if and only if $m\in (0,1)$; otherwise even  eigenfunctions of $\Lambda_{m,X}$ may not belong to $\sH^m(X)$.

\item\label{rem-8-5-1-v}
Since $\Lambda_m$ does not possess transmission property as $m\notin 2\bZ$, we are not in the framework of the Boutet-de-Monvel algebra, but pretty close: $\Lambda_m$ possess $\mu$-transmission property introduced by L.~H\"ormander and systematically studied by G.~Grubb in \cite{grubb:mu-trans, grubb:frac}. We provide definition in Subsection~\ref{sect-8-A-2}.
\end{enumerate}
\end{remark}

We are interested in the asymptotics of the eigenvalue counting function $\N(\lambda)$ for $\Lambda_{m,X}$ as $\lambda\to +\infty$.

\chapter{Preliminary analysis}
\label{sect-8-5-2}
As usual we reduce problem to a semiclassical one.  Let $A=h^m\Lambda_{m,X}-1$ with $h=\lambda^{-1/m}$.

\begin{proposition}\label{prop-8-5-2}
Let $\bar{x}\in X$, $B(\bar{x},2\gamma)\subset X$, $\gamma \ge h$. Then
\begin{equation}
|e(x,x,0) -\Weyl(x) |\le Ch^{1-d}\gamma^{-1}\qquad
\forall x\in B(\bar{x},\gamma)
\label{8-5-2}
\end{equation}
and
\begin{multline}
|\int \psi ((x-\bar{x})/\gamma) \Bigl(e(x,x,0) -\Weyl(x)\Bigr)\,dx| \le \\
Ch^{1-d}\gamma^{d-1}\bigl(\gamma^\delta + h^{\delta}\gamma^{-\delta}\bigr)
\label{8-5-3}
\end{multline}
as $\psi \in \sC^\infty_0 (B(0,1))$ and $\delta>0$.
\end{proposition}

\begin{proof}
Estimate (\ref{8-5-2}) is easily proven  by just rescaling as modulo $O(h^s\gamma^{-s})$ we get a $\hbar$-pseudodifferential operator with $\hbar=h\gamma^{-1}$.

Estimate (\ref{8-5-3}) is easily proven  by rescaling plus R.~Seeley's method as described in Subsection~\ref{book_new-sect-7-5-1} of \cite{futurebook}. We leave easy details to the reader.
\end{proof}

Then we immediately arrive to

\begin{corollary}\label{cor-8-5-3}
\begin{enumerate}[label=(\roman*), wide, labelindent=0pt]
\item\label{cor-8-5-3-i}
Contribution of the inner zone $\{x: \, \dist (x,\partial X)\ge h\}$ to the Weyl remainder does not exceed $Ch^{1-d}$.

\item\label{cor-8-5-3-ii}
Contribution of the intermediate strip
$\{x: \, \varepsilon\ge \dist (x,\partial X)\ge \varepsilon^{-1} h\}$ to the Weyl remainder does not exceed $\eta(\varepsilon)h^{1-d}$ with $\eta(\varepsilon)\to 0$ as $\varepsilon\to 0$.
\end{enumerate}
\end{corollary}

Here and in what follows $\varepsilon>0$ is an arbitrarily small constant.

\begin{proposition}\label{prop-8-5-4}
The following estimate holds:
\begin{equation}
|e(x,x,0)|\le Ch^{-d}.
\label{8-5-4}
\end{equation}
\end{proposition}

\begin{proof}
The standard proof we leave to the reader.
\end{proof}

\begin{theorem}\label{thm-8-5-5}
\begin{enumerate}[label=(\roman*), wide, labelindent=0pt]
\item\label{thm-8-5-5-i}
For operator $A$ Weyl remainder in the asymptotics for $\N^-$ does not exceed $Ch^{1-d}$.
\item\label{thm-8-5-5-ii}
For operator $\Lambda_{m,X}$ the following asymptotics holds
\begin{equation}
\N (\lambda)= \kappa_0 \lambda^{\frac{d}{m}}+ O(\lambda^{\frac{d-1}{m}})
\qquad\text{as\ \ } \lambda\to +\infty
\label{8-5-5}
\end{equation}
where $\kappa_0 =(2\pi)^{-d}\varpi_d \Vol (X)$ and $\Vol (X)$ means the Riemannian volume of $X$.
\end{enumerate}
\end{theorem}

\begin{proof}
Statement~\ref{thm-8-5-5-i} follows immediately from Corollary~\ref{cor-8-5-3}\ref{cor-8-5-3-i} and Proposition~\ref{prop-8-5-4}.
Statement~\ref{thm-8-5-5-ii} follows immediately from \ref{thm-8-5-5-i} as
$d\ge 2$.
\end{proof}

This was  easy but recovering the second term is a much more daunting task requiring first to improve the contribution of the near boundary strip
$\{x:\,\dist (x,\partial X)\le \varepsilon^{-1}h\}$ and also of the inner zone
$\{x:\,\dist (x,\partial X)\ge \varepsilon\}$.

\chapter{Propagation of singularities near boundary}
\label{sect-8-5-3}

Without any loss of the generality one can assume that
\begin{equation}
X=\{x:\, x_1>0\},\qquad g^{jk}=\updelta_{1j}\quad \forall j=1,\ldots, d.
\label{8-5-6}
\end{equation}

First let us study the propagation of singularities along the boundary:

\begin{theorem}\label{thm-8-5-6}
On the energy level $\tau :|\tau|\le \epsilon_0$
\begin{enumerate}[label=(\roman*), wide, labelindent=0pt]
\item\label{thm-8-5-6-i}
$x'$ singularities propagate with the speed not exceeding $c$ with respect to $(x',\xi')$.
\item\label{thm-8-5-6-ii}
As $|\xi'|\asymp \rho\ge C h^{\frac{1}{2}-\delta}$  singularities move from $x'=y'$ with the speed $\asymp \rho$ with respect to $x$.
\end{enumerate}
\end{theorem}

\begin{proof}
The proof is standard as it involves only pseudodifferential operators $p(x,hD')$ and their commutators with $A$ but one can see easily that those commutators do not bring any troubles as the energy level is $\asymp 1$. We leave all easy details to the reader.
\end{proof}

\begin{corollary}\label{cor-8-5-7}
\begin{enumerate}[label=(\roman*), wide, labelindent=0pt]
\item\label{cor-8-5-7-i}
Let $\rho \ge C h^{\frac{1}{2}-\delta}$, $h^{1-\delta}\le \gamma\le \epsilon$. Then the contribution of the zone
$\{(x,\xi'):\,x_1\le \gamma ,|\xi'|\asymp \rho\}$   to the Tauberian remainder with
$T\in (T_*(\rho), T^*(\rho)$, $T_*(\rho)=h^{1-\delta}\rho^{-2}$,
$T^*(\rho)=\epsilon \rho$  does not exceed
\begin{equation}
C\rho^{d-1}h^{-d}\times \gamma \times h^{1-\delta}\rho^{-2} \times \rho^{-1}.
\label{8-5-7}
\end{equation}
\item\label{cor-8-5-7-ii}
The total contribution of the zone
$\{(x,\xi'):\,x_1\le \gamma=h^{1-\delta}\}$ to the Tauberian error with $T=h^{1-3\delta}$ does not exceed $Ch^{-d+1+\delta}$.
\end{enumerate}
\end{corollary}

\begin{proof}
The easy and  standard proof of Statement~\ref{cor-8-5-7-i} is left to the reader.

Then the  contribution of the zone
$x_1\le h^{1-\delta}$, $|\xi'|\ge C h^{\delta}$ to the Tauberian remainder with $T=h^{1-3\delta}$ does not exceed $Ch^{-d+1+\delta}$. Since the contribution of the zone $\{(x,\xi'):\,x_1\le \gamma=h^{1-\delta}, |\xi'|\le C h^{2\delta}\}$ to the asymptotics  does not exceed
$C\rho^{d-1}h^{-d}\times \gamma\le Ch^{-d+1+\delta}$ we arrive to Statement~\ref{cor-8-5-7-ii}.
\end{proof}

Therefore in  this zone $\{(x,\xi'):\,x_1\le \gamma=h^{1-\delta}\}$ all we need is to pass from the Tauberian asymptotics to Weyl ones.
\enlargethispage{1.5\baselineskip}

However the inner zone should be reexamined and we need to describe what happens with the propagation along Hamiltonian trajectory in the zone
$\{(x,\xi):\,x_1\le h^{1-\delta}\}$. We can assume that
$|\xi_1|\ge \varepsilon$ since the measure of the remaining trajectories is small, here $\varepsilon >0 $ is an arbitrarily small constant.

\chapter{Reflection of singularities from the boundary}
\label{sect-8-5-4}

\section{Pilot model}
\label{sect-8-5-4-1}

We start from the pilot-model which will be used to prove the main case. Namely, let us consider $1$-dimensional operator  on half-line $\bR^+$ with Euclidean metrics
\begin{equation}
B\Def B_{m,a,h} = (( h^2 D_x^2+a^2)^{m/2})_\D
\label{8-5-8}
\end{equation}
with $a\ge 0$. We denote $\mathbf{e}_{m,a,h}(x_1,y_1,\tau)$ the Schwartz kernel of its spectral projector.

Observe that scaling $x\mapsto x\gamma^{-1}$, $\tau\mapsto \tau\rho^{-m}$ transforms operator to one with $h\mapsto h/(\rho\gamma)$,
$\tau\mapsto \tau\rho^{-m}$; because of this we can assume that $h=1$ and the second scaling implies that we can assume that either $a=1$ or $\tau=1$.

\begin{proposition}\label{prop-8-5-8}
\begin{enumerate}[label=(\roman*), wide, labelindent=0pt]
\item\label{prop-8-5-8-i}
The spectrum of operator $\Lambda_{m,a}$  is absolutely continuos and it coincides with $[a^m,\infty)$.

\item\label{prop-8-5-8-ii}
The following equalities hold:
\begin{multline}
\mathbf{e}_{m,a,h}(x,y,\lambda)=
\mathbf{e}_{m,1,ha^{-1}}(a x, a y,\lambda a^{-m})=\\[2pt]
\lambda^{1/m}\mathbf{e}_{m,a\lambda^{-1/m},h}(\lambda^{1/m}x, \lambda^{1/m}y,1)
= a\mathbf{e}_{m,1,h}(a x, a y,\lambda a^{-m}).
\label{8-5-9}
\end{multline}
\end{enumerate}
\end{proposition}

\begin{proposition}\label{prop-8-5-9}
Let $\psi \in \sC^\infty _0 ([-1,1])$, $\psi_\gamma (x)=\psi (x/\gamma)$ and $\phi \in \sC_0^\infty([-1,1])$, $0\le a\le 1-\tau0$. Then as
$\gamma\ge h^{1-\delta}$, $T\ge C_0\gamma$,
$h^\delta\ge \eta\ge h^{1-\delta}T^{-1}$
\begin{equation}
\| \phi (\eta^{-1} hD_t-1) \psi_\gamma  e^{i(mh)^{-1}tB}\psi_\gamma|_{t=T} \|
 \le CT^{-1}\gamma + Ch^{\delta'} .
\label{8-5-10}
\end{equation}
\end{proposition}

\begin{proof}
Observe first that if for $u$ supported in $\bR^+$ and $L=x_1 hD_1-ih/2=L^*$
\begin{gather}
\Re i (B L u, u)= \frac{1}{2}  (i[B,L]   u,u)
\label{8-5-11}\\
\shortintertext{and then}
\Re (L u, u_t- ih^{-1}Bu)=\frac{1}{2}(ih^{-1}[B,L]u,u)+\frac{1}{2}\partial _t \Re (Lu,u)
\label{8-5-12}
\end{gather}
and
\begin{multline}
\Re (k tu+ Lu, u_t-ih^{-1}Bu)=\\
\frac{1}{2}\partial_t \bigl(k t\|u\|^2 + (Lu,u)\bigr) +
\frac{1}{2}\bigl( (ih^{-1}[B,L]  u,u)-k\|u\|^2\bigr).
\label{8-5-13}
\end{multline}
Let us plug
\begin{equation}
u=\phi (\eta^{-1} hD_t-1)  e^{ih^{-1}TB}\psi_\gamma  v
\label{8-5-14}
\end{equation}
with $\|v\|=1$. Then the left hand expression in (\ref{8-5-13}) is $0$ and
\begin{equation}
\frac{1}{2}\partial_t \bigl(k t\|u\|^2 + (Lu,u)\bigr)\le
\frac{1}{2}\bigl( -(ih^{-1}[B,L]  u,u)+k\|u\|^2\bigr).
\label{8-5-15}
\end{equation}
Let us estimate from above the right-hand expression; obviously
\begin{equation}
ih^{-1}[B_{m},L]=m (B_m - a^2 B_{m-2}).
\label{8-5-16}
\end{equation}
\begin{enumerate}[label=(\alph*), wide, labelindent=0pt]
\item\label{proof-8-5-9-a}
Assume first that $m>2$. Then since
\begin{gather}
(1-C\eta) \|u\| \le \|B_{m} u\| \le (1+C\eta) \|u\|
\label{8-5-17}\\
\intertext{due to  cutoff by $\phi$ and}
B_{m-2}\le B_{m}^{(m-2)/m}
\label{8-5-18}
\end{gather}
in virtue of Corollary~\ref{cor-8-A-2} we conclude that as $k=m( 1-a^2) $ the right-hand expression does not exceed $C\eta$ and therefore
\begin{equation}
m(1-a^2) t\|u\|^2 + (Lu,u) \le C\gamma +C\eta T
\label{8-5-19}
\end{equation}
since the value of the left-hand expression as $t=0$ does not exceed $C\gamma$.

On the other hand, observe that on the energy levels from $(1-C_0\eta,1+C_0\eta)$\\ the singularities propagate with a speed (with respect to $x_1$) not exceeding $m(1-a^2)^{\frac{1}{2}}(1+C_0\eta)$. Therefore we conclude that $u$ is negligible as ${|x_1|\ge m(1-a^2)^{\frac{1}{2}}(1+C_0\eta) T+C\gamma}$ and therefore since
\begin{gather}
\|D_1 u\|\le (B_2u,u) \le ((1-a^2)^{\frac{1}{2}}+C_0\eta)
\label{8-5-20}\\
\intertext{we conclude using (\ref{8-5-17}) and (\ref{8-5-18}) that}
|(Lu,u)|\le  m (1-a^2+C_0\eta )T+ C\gamma -\epsilon_0 T\|\psi_\gamma (x_1)u\|^2
\label{8-5-21}
\end{gather}
and the left-hand expression of (\ref{8-5-19})  is greater than
$\epsilon_0 T\|\psi_\gamma u\|^2 - C(\eta T+\gamma)$ and we arrive to (\ref{8-5-10}).
\item\label{proof-8-5-9-b}
Assume now that $0<m<2$. Then our above proof fails short in both estimating $\Re ih^{-1}([B,L]u,u)$ from below and $|(Lu,u)|$ from above and we need to remedy it.

First, away from $x_1=0$ only symbols are important and therefore the right-hand expression of (\ref{8-5-15}) does not exceed
$\frac{m}{2}(a^m -a^2) \|\psi_\sigma u\|^2+ C h^2\sigma^{-2}$ since $\|B_{m-2}\|\le a^{m-2}$. Indeed, we need just to decompose
$1=\psi_\sigma^2 + \psi_\sigma^{\prime\,2}$ and use our standard arguments to rewrite the right-hand expression of (\ref{8-5-15}) as the sum of the same expressions for $\psi_\sigma u$ and $\psi'_\sigma u$ plus
$C h^2\sigma^{-2}\|u\|^2\|$.

Similarly we deal
$\Re (Lu,u)=\Re (L\psi_\sigma u, \psi_\sigma u)+
\Re (L\psi'_\sigma u, \psi'_\sigma u)$ and the absolute value of the second term does not exceed
$m(1-a^2)^{\frac{1}{2}}T\|B_m^{1/m}\psi'_\sigma \| ^2$.

We claim that
\begin{equation}
\Re (L\psi_\sigma u, \psi_\sigma u)\le
C\sigma   \|\psi_\sigma u\|^2 + C\sigma h^\delta;
\label{8-5-22}
\end{equation}
we prove it late but now instead of (\ref{8-5-19}) we arrive as
$\sigma = \epsilon_0t$ to
\begin{equation*}
((1-a^2)-\epsilon) P(T) \le
Ch^{\delta'} +   (a^m-a^2) T^{-1} \int_\gamma^T P(t)\,dt +C\gamma T^{-1}
\end{equation*}
with $P(t)=\|\psi_{\epsilon_0t} u(.,t)\|^2$. Then
since $\nu =(a^m-a^2)/((1-a^2)-\epsilon)<0$ which implies (\ref{8-5-10}) again.
\end{enumerate}
\end{proof}

\begin{proof}[Proof of \textup{(\ref{8-5-22})}]
Indeed, as $h=1$, $\|B_m^{1/m} u\|\le 1$ we from G.~Grubb~\cite{grubb:mu-trans, grubb:frac} conclude that
$|u(x_1)|\le Cx_1^{(m-1)/2}\|u B_m^{1/m}u\|$ and
$|Lu(x_1)|\le C x_1^{(m-1)/2}$ and therefore $|(Lu,u)|\le C\sigma^m \|B_m^{1/m}u\|^2$. Take $\sigma=1$.

Scaling returns (\ref{8-5-22}) as $\sigma=h$.
\end{proof}

\begin{proposition}\label{prop-8-5-10}
Let  $\Lambda=\Lambda_{m,X}$ be a  $d$-dimensional operator \textup{(\ref{8-5-1})} on the half-space
$X=\{x\in \bR^d, x_1>0\}$ with Euclidean metrics ($d\ge 2$) and $A=h^m\Lambda^m-1$.

Let $\psi \in \sC^\infty _0 ([-1,1])$, $\psi_\gamma (x)=\psi (x_1/\gamma)$,
$\phi \in \sC_0^\infty([-1,1])$, $\varphi \in \sC_0^\infty(\bR^{d-1})$ supported in $\{|\xi'|\le 1-\epsilon\}$ with $\epsilon>0$. Finally, let
$\gamma\ge h^{1-\delta}$, $T\ge C h^{-\delta}\gamma$,
$h^\delta\ge \eta\ge h^{1-\delta}T^{-1}$. Then
\begin{equation}
\|\phi (\eta^{-1} hD_t) \varphi (hD') \psi_\gamma (x_1) e^{ih^{-1}tA}\psi|_{t=T}\| = O(h^s)
\label{8-5-23}
\end{equation}
with arbitrarily large $s$.
\end{proposition}

\begin{proof}
By making Fourier transform $F_{x'\to h^{-1}\xi'}$ we reduce the general case to $d=1$ and operator $B$.

According to Proposition~\ref{prop-8-5-9}
$\|\phi (\eta^{-1} hD_t)  \psi _\gamma e^{ih^{-1}TA}\psi\|\le h^\delta$.  Thus for $s=\delta'$ (\ref{8-5-23}) has been proven (we reduce $\delta'$ if necessary).

Without any loss of the generality we assume that $\psi (x_1)=1$ as $x_1\le 1$, $\psi (x_1) =0$ as $x_1\ge 2$.

Observe that due to propagation as $t\le T$ and $x_1\ge \gamma$ we see that
$\phi (\eta^{-1} hD_t)  Q^+(hD_1)(1-\psi_\gamma) e^{ih^{-1}TA}\psi_\gamma$ is negligible where $Q^\pm\in \sC^\infty (\bR)$ is supported in
$\{\pm \xi_1> \epsilon \}$.

Furthermore, from the standard ellipticity arguments we conclude that
$\phi (\eta^{-1} hD_t)  Q^0(hD_1)(1-\psi_\gamma) e^{ih^{-1}TA}\psi_\gamma$ is also negligible for $Q^0\in \sC_0^\infty ([-2\epsilon,2\epsilon])$.

Finally, due to propagation as $t\ge T$ and $x_1\ge \gamma$ we conclude that
$\phi (\eta^{-1} hD_t)  \psi_\gamma e^{ih^{-1}(t-T)A}Q^-(hD_1)(1-\psi_\gamma)
e^{ih^{-1}TA }\psi_\gamma$ is negligible as $t\ge T$.

What is left is   $\phi (\eta^{-1} hD_t)  \psi_\gamma e^{ih^{-1}(t-T)A}Q^-(hD_1)\psi_\gamma e^{ih^{-1}TA }\psi_\gamma$ and since (\ref{8-5-23}) holds for $s=\delta'$ we conclude that it holds for $s=2\delta'$ and $T$ replaced by $2T$.

Continuing this process we see that (\ref{8-5-23}) holds for $s=n\delta'$ and $T$ replaced by $nT$. Therefore, as we redenote $nT$ by $T$ (and $T$ by $T/n$ respectively), we acquire factor $(\gamma n/T)^n$ in our estimate and it is $O(h^s)$ for any $s$ as $h$ is sufficiently small, $\gamma/T \le h^\delta$ and $n=s/\delta'$.
\end{proof}

\section{General case}
\label{sect-8-5-4-2}

\begin{theorem}\label{thm-8-5-11}
Let $(\bar{x},\bar{\xi})$ be a point on the energy level $1$. Consider a Hamiltonian trajectory $\Psi_t (\bar{x},\bar{\xi})$ with $\pm t\in [0, m T]$ (one sign only) with $T\ge \epsilon_0$  and assume that for each $t$ indicated it meets $\partial X$ transversally i.e.
\begin{multline}
\dist (\uppi_x \Psi_t (x,\xi),\partial X) \le \epsilon
\implies\\
 |\frac{d\ }{dt}\dist (\uppi_x \Psi_t (x,\xi),\partial X)|\ge \epsilon \qquad \forall t: \pm t\in [0, mT].
\label{8-5-24}
\end{multline}
Also assume that
\begin{equation}
\dist (\uppi_x \Psi_t (x,\xi),\partial X)\ge \epsilon_0 \qquad \text{as\ \ }
t=0, \ \pm t=mT.
\label{8-5-25}
\end{equation}
Let $\epsilon >0$ be a small enough constant, $Q$ be supported in $\epsilon$-vicinity of $(x,\xi)$ and $Q_1\equiv 1$ in $C_0\epsilon$-vicinity of $\Psi_t(x,\xi)$ as $t=\pm mT$. Then operator  $(I-Q_1) e^{-ih^{-1}tH}Q$ is negligible as $t=\pm mT$.
\end{theorem}

\begin{proof}
\begin{enumerate}[label=(\alph*), wide, labelindent=0pt]
\item\label{proof-8-5-11-a}
Obviously without any loss of the generality one can assume that there is just one reflection from $\partial X$ (and this reflection is transversal) and that (\ref{8-5-6}) is fulfilled in its vicinity.

Further,  without any loss of the generality one can assume that $Q$ is supported in $\varepsilon$-vicinity of $(\bar{x},\bar{\xi})$, $Q_1\equiv 1$ in  $\varepsilon$-vicinity of $\Psi_{mT} (\bar{x},\bar{\xi})$ and
$T\asymp \varepsilon$ with $\varepsilon=h^{\frac{1}{2}-\delta'}$. Then both $\bar{x}$ and $\uppi_x \Psi_{mT} (\bar{x},\bar{\xi})$ belong to $C_0\varepsilon$-vicinity of $\partial X$.

Indeed, it follows from the propagation inside of domain.

\item\label{proof-8-5-11-b}
Then instead of isotropic vicinities we can consider anisotropic ones: $\varepsilon$ with respect to $(x',\xi')$, $h^{1-3\delta'}$ with respect to $x_1$ and $h^{\delta'}$ with respect to $\xi_1$. Let now $Q$ and $Q_1$ be corresponding operators.

In this framework from the propagation inside of domain it follows that without any loss of the generality one can assume that
$T\asymp \gamma=h^{1-\delta''}$ and both $\bar{x}$ and
$\uppi_x \Psi_{mT} (\bar{x},\bar{\xi})$ belong to $C_0\gamma$-vicinity of $\partial X$.

\item\label{proof-8-5-11-c}
Then one can employ the method of the successive approximations freezing coefficients at point $\bar{x}$ and in this case the statement of the theorem follows from the construction of Section~\ref{book_new-sect-7-2} of \cite{futurebook}\footnote{\label{foot-8-23} Insignificant and rather obvious modifications are required.}   and  Proposition~\ref{prop-8-5-10}. We leave easy details to the reader.
\end{enumerate}
\end{proof}

Then we arrive immediately to
\begin{corollary}\label{cor-8-5-12}
Under standard non-periodicity condition  $\N^-(h)$  is given with $o(h^{1-d})$-error by the Tauberian expression with $T=h^{1-\delta}$.
\end{corollary}

\begin{proof}
Easy details are left to the reader.
\end{proof}

\chapter{Main results}
\label{sect-8-5-5}

\section{From Tauberian to Weyl asymptotics}
\label{sect-8-5-5-1}

Now we can apply the method of successive approximations as described in Section~\ref{book_new-sect-7-2}\,\footref{foot-8-23} and prove that for operator $A$ the Tauberian expression with $T=h^{1-\delta}$ (with sufficiently small $\delta>0$) equals to Weyl expression $\N^\W_h$ with $O(h^{2-d-\delta''})$ error,
\begin{equation}
\N^\W_h = \kappa_0 h^{-d} + \kappa_{1,m} h^{1-d}+o(h^{1-m})
\label{8-5-26}
\end{equation}
with the standard coefficient $\kappa_0 =(2\pi)^{-d}\varpi_d\Vol_d (X)$ and with
\begin{equation}
\kappa_{1,m}= (2\pi)^{1-d}\varpi_{d-1} \varkappa_{m}\Vol_{d-1}(\partial X)
\label{8-5-27}
\end{equation}
where
\begin{multline}
\varkappa_m  =\\
=\frac{d-1}{m}
\iint_1^\infty \lambda ^{-(d-1)/m-1}
\Bigl( \mathbf{e}_m (x_1,x_1,\lambda) - \pi ^{-1}(\lambda-1)^{1/m}\Bigr)\,dx_1d \lambda
\label{8-5-28}
\end{multline}
with $\mathbf{e}_{m}(x_1,y_1,\tau)=\mathbf{e}_{m,1,1}(x_1,y_1,\tau)$ the Schwartz kernel of the spectral projector of operator
$\mathbf{a}_m\Def B_{m,1,1}$ introduced by (\ref{8-5-8}):
\begin{equation}
\mathbf{a}_m = ((  D_x^2+1)^{m/2})_\D
\label{8-5-29}
\end{equation}

Recall that $\Vol_d$ and $\Vol_{d-1}$ are Riemannian volumes corresponding to metrics $g$ and its restriction to $\partial X$ respectively and
$\pi ^{-1}(\lambda-1)^{1/m}$ is a Weyl approximation to $\mathbf{e}_{m,1}(x_1,x_1,\lambda)$.

Thus we arrive to
\begin{theorem}\label{thm-8-5-13}
Under standard non-periodicity condition the following asymptotics holds:
\begin{equation}
\N(\tau) = \kappa_0 \tau^{\frac{d}{m}} + \kappa_{1,m} \tau^{\frac{d-1}{m}} +o(\tau^{\frac{d-1}{m}}) \qquad\text{as\ \ }\tau\to +\infty.
\label{8-5-30}
\end{equation}
\end{theorem}

\begin{proof}
First we establish as described above asymptotics
\begin{equation}
\N^-_h = \kappa_0 h^{-d} + \kappa_{1,m} h^{1-d} +o(h^{1-d}) \qquad
\text{as\ \ }h\to +0
\label{8-5-31}
\end{equation}
which immediately implies (\ref{8-5-30}) .
\end{proof}

\section{Discussion}
\label{sect-8-5-5-2}

The following problems seem to be challenging

\begin{Problem}\label{Problem-8-5-14}
As $m_1>0$, $m_2>0$ consider
\begin{equation}
K=K_{m_1,m_2,X}\Def \Lambda_{m,X}-\Lambda_{m_1,X}\Lambda_{m_2,X}
\label{8-5-32}
\end{equation}
on $\fD(\Lambda_m)$ with $m=m_1+m_2$. From Corollary~\ref{cor-8-A-2} we conclude that this is non-negative operator; obviously singularities of its Schwartz kernel $K(x,y)$ belong to $\partial X \times \partial X$. Prove that
\begin{enumerate}[label=(\roman*), wide, labelindent=0pt]
\item\label{Problem-8-5-14-i}
$K$ a positive operator.
\item\label{Problem-8-5-14-ii}
As $X=\{x\in \bR^d:\,x_1>0\}$ with Euclidean metrics its Schwartz kernel $K(x,y)= k(x_1,y_1,x'-y')$ which is positive homogeneous of degree $-m-d$ satisfies
\begin{equation}
|D^\alpha_{x} D^\beta_{y}K(x,y)|\le
C_{\alpha\beta} (x_1+y_1)^{-m-\alpha_1-\beta_1} (x_1+y_1+|z|)^{-d-|\alpha'|+|\beta'|}.
\label{8-5-33}
\end{equation}
\item\label{Problem-8-5-14-iii}
In the general case in the local coordinates in which $X=\{x:\,x_1>0\}$  and $x_1=\dist(x,\partial X)$  not only (\ref{8-5-33}) holds but also
\begin{multline}
|D^\alpha_{x} D^\beta_{y}\bigl(K(x,y)-K^0(x,y)\bigr)|\\
\le
C_{\alpha\beta} (x_1+y_1)^{-m-\alpha_1-\beta_1} (x_1+y_1+|z|)^{-d-|\alpha'|+|\beta'|+1}
\label{8-5-34}
\end{multline}
where $K^0(x,y)= k(x_1,y_1, x'-y')$ and $g^{jk}=\updelta_{jk}$ at point $(0,\frac{1}{2}(x'+y')$.
\end{enumerate}
\end{Problem}

\begin{Problem}\label{Problem-8-5-15}
As $m>0$, $n>0$ consider operator
\begin{equation}
K =K_{m,n,X}\Def \Lambda_{m,X}- \Lambda_{n,X}^{m/n}
\label{8-5-35}
\end{equation}
on $\fD(\Lambda_k)$ with $k=\max(m,n)$. Then this is non-negative (no-positive) operator as $m>n$ ($m<n$ respectively). Prove that
\begin{enumerate}[label=(\roman*), wide, labelindent=0pt]
\item\label{Problem-8-5-15-i}
$K$ a positive (negative) operator as $n>m$ ($n<m$ respectively).

\item\label{Problem-8-5-15-ii}
As $X=\{x\in \bR^d:\,x_1>0\}$ with Euclidean metrics its Schwartz kernel $K(x,y)= k(x_1,y_1,x'-y')$ which is positive homogeneous of degree $-m-d$ satisfies (\ref{8-5-33}).

\item\label{Problem-8-5-15-iii}
In the framework of Problem~\ref{Problem-8-5-14}\ref{Problem-8-5-14-iii} both
(\ref{8-5-33}) and (\ref{8-5-34}) hold.
\end{enumerate}
\end{Problem}

\begin{Problem}\label{Problem-8-5-16}
\begin{enumerate}[label=(\roman*), wide, labelindent=0pt]
\item\label{Problem-8-5-16-i}
Consider operators $(\Delta^{m/2})_\D$ with $m<0$ and the asymptotics of eigenvalues tending to $+0$.

\item\label{Problem-8-5-16-ii}
Consider operators with degenerations like $A_{m,X}= h^m \Lambda_{m,X} + V(x)$.

\item\label{Problem-8-5-16-iii}
Consider more general operators where instead of $\Delta$ general elliptic (matrix) operator is used.
\end{enumerate}
\end{Problem}

\begin{Problem}\label{Problem-8-5-17}
\begin{enumerate}[label=(\roman*), wide, labelindent=0pt]
\item\label{Problem-8-5-17-i}
Consider Neumann boundary conditions: having smooth metrics $g$ in the vicinity of $\bar{X}$ for each point $x\notin X$ in the vicinity of $\partial X$ we can assign a mirror point $j(x)\in \bar{X}$ such that $x$ and $j(x)$ are connected by a (short) geodesics orthogonal to $Y$ at the point of intersection. Each $u$ defined in $X$ we can continue to the vicinity of $\bar{X}$ as
$ J u(x)=\psi (x) u(j(x))$ with $\psi$ supported in the vicinity of $\bar{X}$ and $\psi=1$ in the smaller vicinity of $\bar{X}$. Then
$\Lambda_m u= R_X \Delta^{m/2} Ju$.
\begin{itemize}[label=-]
\item
Establish eigenvalue asymptotics for this operator.

\item
Surely we need to prove that the choice neither of metrics outside of $X$ nor $\psi$ is important.
\end{itemize}
\item\label{Problem-8-5-17-ii}
One can also try $J u(x)=-\psi (x)u(j(x))$ and prove that eigenvalue asymptotics for this operator do not differ from what we got just for continuation by $0$.
\end{enumerate}
\end{Problem}

\begin{Problem}\label{Problem-8-5-18}
Consider manifolds with all geodesic billiards closed as in Section~\ref{book_new-sect-8-3} of~\cite{futurebook}. To do this we need to calculate the ``phase shift'' at the transversal reflection point itself seems to be an extremely challenging problem.
\end{Problem}

\begin{subappendices}
\chapter{Appendices}
\label{sect-8-A}
\section{Variational estimates for fractional Laplacian}
\label{sect-8-A-1}

We follow here R.~Frank  and  L.~Geisinger~\cite{frank:gei:2}. This is Lemma~19  and the next paragraph of their paper:

\begin{lemma}\label{lemma-8-A-1}
\begin{enumerate}[label=(\roman*), wide, labelindent=0pt]
\item\label{lemma-8-A-5-i}
Let $B$ be a non-negative operator with $\Ker B=\{0\}$ and let $P$ be an orthogonal projection. Then for any operator monotone function $\phi:(0,\infty)\to \bR$,
\begin{equation}
\label{8-A-1}
P\phi(PBP)P \ge P\phi(B)P .
\end{equation}
\item\label{lemma-8-A-5-ii}
If, in addition, $B$ is positive definite and $\phi$ is not affine linear, then $\phi(PBP)=P\phi(B)P$ implies that the range of $P$ is a reducing subspace of $B$.
\end{enumerate}
\end{lemma}

We recall that, by definition, the range of $P$ is a reducing subspace of a non-negative (possibly unbounded) operator if $(B+\tau)^{-1}\Ran P\subset\Ran P$ for some $\tau>0$. We note that this is equivalent to $(B+\tau)^{-1}$ commuting with $P$, and we see that the definition is independent of $\tau$ since
\begin{multline*}
(B+\tau')^{-1} P - P (B+\tau')^{-1} \\
= (B+\tau) (B+\tau')^{-1} \left( (B+\tau)^{-1} P - P (B+\tau)^{-1} \right) (B+\tau) (B+\tau')^{-1} .
\end{multline*}

We refer to the proof given there.

\begin{corollary}\label{cor-8-A-2}
The following inequality holds
\begin{equation}
\Lambda_{m,X} \le \Lambda_{n,X}^{m/n} \qquad \text{as\ \ }0<m<n.
\label{8-A-2}
\end{equation}
\end{corollary}

\begin{proof}
Plugging into (\ref{8-A-1})  $B = \Delta^{n/2}$ in $\bR^d$,  $P=\theta_X(x)$ and  $\phi(\lambda) = \lambda^{m/n}$ we get (\ref{8-A-2}).
\end{proof}

Repeating arguments of Proposition~20 and following it Subsection 6.4 of R.~Frank  and  L.~Geisinger~\cite{frank:gei:2} (powers of operators will be different but also negative) we conclude that

\begin{proposition}\label{prop-8-A-3}
Let $d\ge 2$. Then   $-\varkappa_m$ is positive strictly monotone increasing function of $m>0$.
\end{proposition}

We leave details to the reader.

\section{$\mu$-transmission property}
\label{sect-8-A-2}

Proposition~1 of G.~Grubb~\cite{grubb:frac} claims that

\begin{proposition}\label{prop-8-A-4}
A necessary and sufficient condition in order that
$R_X Pu\in \sC^\infty (\bar{X})$ for all $u\in \sE_\mu(\bar{X})$ is that $P$ satisfies the \emph{$\mu $-transmission condition\/} (in short: is of type
$\mu $), namely  that
\begin{equation}
\partial_x^\beta \partial_\xi ^\alpha {p_j}(x,-N)=
e^{\pi i(m-2\mu -j-|\alpha | )} \partial_x^\beta \partial_\xi ^\alpha{p_j}(x,N)  \qquad\forall x\in\partial\Omega,
\label{8-A-3}
\end{equation}
for all $j,\alpha ,\beta $, where $N$ denotes the interior normal to
$\partial X $ at $x$, $m$ is an order of classical pseudo-differential operator $P$ and for $\mu \in \bC$ with $\Re\mu >-1$.
\end{proposition}

Here  $\sE_\mu (\bar{X})$ denotes the space of functions $u$ such that
$u=E_X d(x)^\mu v$ with  $v\in \sC^\infty (\bar{X})$ where $E_X$ is an operator of extension by $0$ to $\bR^d\setminus X$ and $d(x)=\dist (x,\partial X)$.

Observe that for $\mu=0$ we have an ordinary transmission property (see Definition \ref{book_new-def-1-4-3}).
\end{subappendices}

\chapter{Global theory}
\label{sect-11-3-7}

Let us discuss fractional Laplacians defined by (\ref{8-5-1}) in domain $X\subset \bR^d$. Then under additional condition
\begin{equation}
\dist (x,y) \le C_0 |x-y|\qquad \forall x,y \in X
\label{11-3-54}
\end{equation}
(where $\dist(x,y)$ is a ``connected'' distance between $x$ and $y$) everything seems to work. We leave to the reader.

\begin{Problem}\label{Problem-11-3-37}
Under assumption (\ref{book_new-11-3-54}) of \cite{futurebook}
\begin{enumerate}[label=(\roman*), wide, labelindent=0pt]
\item\label{Problem-11-3-37-i}
prove Lieb-Cwikel-Rozeblioum estimate (\ref{book_new-9-A-11}) of \cite{futurebook}.

\item\label{Problem-11-3-37-ii}
Restore results of Chapter~\ref{book_new-sect-9} of \cite{futurebook}.

\item\label{Problem-11-3-37-iii}
Reconsider examples of Sections~\ref{book_new-sect-11-2} and \ref{book_new-sect-11-3} of \cite{futurebook}.
\end{enumerate}
\end{Problem}

\begin{remark}\label{rem-11-3-38}
Obviously domains with cuts and inner spikes (inner angles of $2\pi$)
do not fit (\ref{11-3-54}). On the other hand in the case of the domain with the cut due to non-locality of $\Delta^r$ with
$r\in \mathbb{R}^+\setminus \mathbb{Z}$ both sides of the cut ``interact'' and at least coefficient in the second term of two-term asymptotics may be wrong; in the case of the inner spike some milder effects are expected.

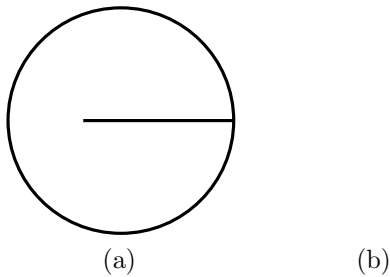
\begin{figure}[h!]
\centering
\subfloat[]{%
\begin{tikzpicture}
\draw[very thick] (0,0) circle (1.5);
\draw[very thick] (-.5,0)--(1.5,0);
\end{tikzpicture}}\qquad\qquad
\subfloat[]{%
\begin{tikzpicture}
\draw[very thick, domain=-pi:pi,scale=1, smooth] plot[parametric,id=parametric-example]
 function{3*cos(t/2)*cos(t), 2*cos(t/2)*sin(t)};
\end{tikzpicture}}
\caption{\label{fig-11-7} Domain with a cut (a) and an inner spike (b).}
\end{figure}

\end{remark}

The following problem seems to be very challenging:

\begin{Problem}\label{Problem-11-3-39}
\begin{enumerate}[label=(\roman*), wide, labelindent=0pt]
\item\label{Problem-11-3-39-i}
Investigate fractional Laplacians in domains with cuts and inner spikes and save whatever is possible.
\item\label{Problem-11-3-39-ii}
Generalize these results to higher dimensions.
\end{enumerate}
\end{Problem}

\chapter*{Comments}

To my surprise I learned that fractional Laplacians are of the interest to probability theory:  which seem to be of interest to probability theory starting from R.~M.~Blumenthal,~R. M. and R.~K.~Getoor~\cite{blum:get} and then by R.~Ba{\~n}uelos and T.~Kulczycki~\cite{ban:kul}, R.~Ba{\~n}uelos, T.~Kulczycki and B.~Siudeja.~\cite{ban:kul:siu}, M.~Kwa\'snicki \cite{kwas}.

Those operators were formulated in the framework of stochastic processes and thus were not accessible for me until I found paper R.~Frank  and  L.~Geisinger~\cite{frank:gei:2} provided definition we follow here. They showed that the trace  has a two-term expansion  regardless of dynamical assumptions\footnote{\label{foot-8-24} The fact that R.~Frank and L.~Geisinger obtain a second term regardless of dynamical assumptions is simply due to the fact that they study $\Tr (f(\Lambda_{m,X}))$ with
$f(\lambda) = -\lambda \uptheta(-\lambda)$, which is one order smoother than $f(\lambda) = \uptheta(-\lambda)$.}, and the second term in their expansion paper~\cite{frank:gei:2} defined by (3.2)--(3.3) is closely related to $\kappa_{1,m}$.

Furthermore I learned that one-term asymptotics for more general operators (albeit without remainder estimate) was obtained by G.~Grubb~\cite{grubb:frac-sp}.

I express my gratitudes to G.~Grubb and R.~Frank for pointing  to rather nasty errors in the previous version and  very useful comments
 and R.~Ba{\~n}uelos  for very useful comments.

\input Preprint-28.bbl
\end{document}